\documentclass[letterpaper, 10 pt, conference]{ieeeconf}

\IEEEoverridecommandlockouts                              
 \pdfoutput=1
\usepackage{amsmath, amssymb}
\usepackage{amsfonts}
\usepackage{graphicx}
\usepackage{enumerate}
\usepackage{caption}
\usepackage{ifthen}
\usepackage{multicol}
\usepackage{float}
\usepackage{cite}
\usepackage[usenames, dvipsnames]{color}
\usepackage[lofdepth,lotdepth]{subfig}
\graphicspath{{figures/}}
\usepackage{url}

\newtheorem{theorem}{Theorem}[section]
\newtheorem{lemma}[theorem]{Lemma}
\newtheorem{proposition}[theorem]{Proposition}
\newtheorem{corollary}[theorem]{Corollary}
\newtheorem{rem}{Remark}

\newtheorem{assumption}{Assumption}

\newboolean{showcomments}
\setboolean{showcomments}{true}

\newcommand{\henrik}[1]{\ifthenelse{\boolean{showcomments}}
{\textcolor{Blue}{Henrik says: #1}}{}}
\newcommand{\emma}[1]{\ifthenelse{\boolean{showcomments}}
{\textcolor{Green}{Emma says: #1}}{}}
\newcommand{\martin}[1]{\ifthenelse{\boolean{showcomments}}
{\textcolor{Skyblue}{Martin says: #1}}{}}
\newcommand{\john}[1]{\ifthenelse{\boolean{showcomments}}
{\textcolor{Red}{John says: #1}}{}}

\newboolean{shownew}
\setboolean{shownew}{true}
\newcommand{\newtext}[1]{\ifthenelse{\boolean{shownew}}
{\textcolor{Red}{#1}}{}}

\newcommand{\hn}{$\mathcal{H}_2$ }

\newcommand{\tr}[1]{\ensuremath{#1^\top}}
\newcommand{\mat}[1]{#1}
\newcommand{\inv}[1]{\mat{#1}^{\raisebox{.2ex}{$\scriptscriptstyle-1$}}}

\newcommand{\DAPI}{\mathrm{DAPI}}
\newcommand{\CAPI}{\mathrm{CAPI}}

\usepackage{array}
\newcolumntype{L}[1]{>{\raggedright\let\newline\\\arraybackslash\hspace{0pt}}m{#1}}
\newcolumntype{C}[1]{>{\centering\let\newline\\\arraybackslash\hspace{0pt}}m{#1}}
\newcolumntype{R}[1]{>{\raggedleft\let\newline\\\arraybackslash\hspace{0pt}}m{#1}}
\usepackage{subfig}

 \usepackage{tikz}
 \usetikzlibrary{decorations.pathmorphing,shadings}
 \usetikzlibrary{plotmarks}
 \usepackage{pgfplots}
\usepackage{tkz-graph}

\definecolor{gray3}{rgb}{0.75, 0.75, 0.75}
\definecolor{gray2}{rgb}{0.5, 0.5, 0.5}
\definecolor{gray1}{rgb}{0.25, 0.25, 0.25}
\definecolor{gray0}{rgb}{0.15, 0.15, 0.15}

\pgfplotscreateplotcyclelist{mygraylist}{%
black, densely dashed\\%
gray1\\%
gray3\\%
gray2\\%
gray2,dashed\\%
gray3,dashed\\%
}

\pgfplotscreateplotcyclelist{myothergraylist}{%
black, densely dashed\\%
gray1\\%
gray2\\%
gray3\\%
gray1, dashed\\%
gray3,dashed\\%
}

\usetikzlibrary{arrows}

\title{\LARGE \bf  Performance Limitations of Distributed Integral Control \\in Power Networks Under Noisy Measurements}

\author{Hendrik Flamme, Emma Tegling and Henrik Sandberg %
\thanks{The authors are with the School of Electrical Engineering and the ACCESS Linnaeus Center, KTH Royal Institute of Technology, SE-100 44 Stockholm, Sweden {(\tt flamme, tegling, hsan@kth.se)}.}  \thanks{ Corresponding author: E. Tegling.} \thanks{This work was supported in part by the Swedish Research Council through grants 2013-5523 and 2016-00861.}
 }%

\usepackage{scalefnt}
        
\begin{document}
\maketitle
\thispagestyle{empty}
\pagestyle{empty}

\begin{abstract}
Distributed approaches to secondary frequency control have become a way to address the need for more flexible control schemes in power networks with increasingly distributed generation. 
The distributed averaging proportional-integral (DAPI) controller presents one such approach. In this paper, we analyze the transient performance of this controller, and specifically address the question of its performance under noisy frequency measurements. Performance is analyzed in terms of an \hn norm metric that quantifies power losses incurred in the synchronization transient. While previous studies have shown that the DAPI controller performs well, in particular in sparse networks and compared to a centralized averaging PI (CAPI) controller, our results prove that additive measurement noise may have a significant negative impact on its performance and scalability. This impact is shown to decrease with an increased inter-nodal alignment of the controllers' integral states, either through increased gains or increased connectivity. For very large and sparse networks, however, the requirement for inter-nodal alignment is so large that a CAPI approach may be preferable. Overall, our results show that distributed secondary frequency control through DAPI is possible and may perform well also under noisy measurements, but requires careful tuning.

\end{abstract}


\section{Introduction}
Many of today's electric power networks are undergoing a paradigm shift, where local, small-scale generation resources are increasingly replacing large-scale centralized power plants~\cite{Farhangi2010, Milligan2015}. This has motivated considerable research efforts in developing more flexible and scalable distributed schemes for generation planning and frequency control. For example, new optimization techniques have been proposed in~\cite {Zhao2014Design, Mallada2014Optimal,Zhang2015realtime, Li2016}, which rather than the traditional economic dispatch exploit the system's frequency dynamics and the possibility of load-side frequency control. 

A second line of research has focused on secondary frequency control, the objective of which is to restore the system's nominal frequency by adjusting generator set-points. This is traditionally handled by a centralized automatic generation control (AGC), but integral control strategies with various degrees of de-centralization have been proposed in the last years~\cite{Shafiee2014, Andreasson2014ACC, SimpsonPorco2013, Dorfler2014, ZhaoMallada2015, Trip2016Internal}. With the right architecture, these controllers can also guarantee optimality in the generators' power injections~\cite{SimpsonPorco2013, Andreasson2014ACC}
, further motivating their use in a distributed generation setting. 

In this paper, we will focus on one such integral controller that solves the secondary frequency control problem in a distributed fashion. It appends to each local frequency droop controller an integral controller, and combines it with an alignment of the integral state with neighboring nodes through a distributed averaging filter. This \emph{distributed averaging proportional-integral (DAPI)} controller has been advocated in e.g.~\cite{SimpsonPorco2013, Andreasson2014ACC, Andreasson2014TAC} and has since then been analyzed with regards to optimal design in~\cite{Wu2016} and for performance in~\cite{Tegling2016ACC, Tegling2017, Andreasson2017CDC}. 

This latter series of work has in particular shown that the DAPI controller improves the system's transient performance compared to standard droop control, and also outperforms the \emph{centralized averaging PI (CAPI)} controller; a centralized integral controller resembling the traditional AGC. This partly counterintuitive result has been attributed to the DAPI controller's ability to detect and attenuate local frequency deviations locally, preventing them from giving rise to large non-equilibrium power flows~\cite{Tegling2016ACC}. 
An important question that arises is therefore what happens to the DAPI controller's superior performance if these local frequency measurements are not perfect, but subject to measurement noise. 

This question has motivated the present work, where we investigate the impact of noisy frequency measurements on   performance. In line with~\cite{Tegling2016ACC}, we evaluate performance in terms of resistive power losses that are incurred in regulating the frequency, when the system is subject to persistent small disturbances due to power imbalances. These transient power losses can be quantified through an \hn norm of an input-output system that describes the generator dynamics, with a suitably defined performance output.  Our evaluation shows that the explicit inclusion of frequency measurement noise has a significant impact on the relative performance of the DAPI and the CAPI controllers. 

It turns out that the measurement noise gives rise to its own contribution to the DAPI controlled system's expected transient power losses. For many control designs, this contribution will be smaller than the losses arising due to power imbalances, and simply lead to a small shift in the optimal controller configuration. However, we show that the noise contribution has an unfavorable scaling with network size. This affects the scalability of the DAPI controller, which in large sparse networks may incur much larger losses than the CAPI controller, whose performance is unaffected by the measurement noise. Still, we show that it is always possible to tune the DAPI controller so that it performs better than the CAPI controller for any given network. Overall, this means that the DAPI controller must be very carefully tuned if it is subject to measurement noise, and this tuning may need to be adjusted subject to the scaling of the network.

Most of the remainder of this paper is devoted to analyzing the additional transient power losses that arise due to measurement noise and its dependence on properties of the network and the controller. After introducing the problem setup in Section~\ref{sec:setup}, we present and analyze the closed-form formulae for the performance of both the DAPI and CAPI controllers with and without measurement noise in Section~\ref{sec:noise}. In Section~\ref{sec:scalability} we discuss the scaling of the DAPI-controller's losses, and prove that even though it may be unfavorable, an adjustment of the distributed averaging filter gains can mitigate the loss increase. We conclude the paper in Section~\ref{sec:conclusions}.

\section{Problem Setup}
\label{sec:setup}
We now introduce the linearized dynamics for the droop-controlled power network, along with the two secondary frequency controllers; DAPI and CAPI. We also introduce a model for the measurement noise arising in these controllers. 

\subsection{Definitions}
Consider a network modeled by the symmetric, weighted, connected graph $\mathcal{G}^P = \{\mathcal{V}, \mathcal{E}^P\}$, where $\mathcal{V} = \{1, \ldots, n\}$ is the set of nodes and $\mathcal{E}^P \subset \mathcal{V} \times \mathcal{V} $  
is the set of edges, or  network lines. Each network line has an associated constant admittance $y_{ji} = y_{ij} = g_{ij} - \mathbf{j}b_{ij}$, where $g_{ij}, b_{ij} > 0$ can be regarded as edge weights. Let $\mathcal{N}_i^P:=\{j\in\mathcal{V}:(i,j)\in\mathcal{E}^P\}$ be the set of all nodes that are connected to node~$i$ in the graph $\mathcal{G}^P$. Then, by definition, $b_{ij}=g_{ij}=0~\forall j\not\in\mathcal{N}_i^P$.  In this paper, we consider a Kron-reduced network model (see e.g. \cite{Varaiya1985}), {in which constant-impedance loads have been eliminated through a reduction procedure and their effect absorbed into} the network line models in~$\mathcal{E}^P$. Consequently, every node $i \in \mathcal{V}$ represents a generation unit (or synchronous load) with an associated phase angle~$\theta_i$ and voltage magnitude $|V_i|$.

For a graph $\mathcal{G} = \{\mathcal{V}, \mathcal{E}\}$ where each edge $(i,j)\in \mathcal{E}$ has an associated weight $\mathrm{w}_{ij} = \mathrm{w}_{ji}>0$, we define the weighted graph Laplacian matrix $\mathcal{L}_W$ through $[\mathcal{L}_W]_{ij} = -\mathrm{w}_{ij}$ if $i \neq j$ and $[\mathcal{L}_W]_{ii} = \sum_{j \in \mathcal{N}_i} \mathrm{w}_{ij}$. Note that the Laplacian $\mathcal{L}_W$ is positive semidefinite.

We denote the transpose of an arbitrary matrix $A$ by $\tr{A}$. The identity matrix is denoted $I$, and a column vector with all components equal to $1$ is denoted $\mathbf{1}$. 

\subsection{Droop controlled power network}
Under standard droop control, each generator $i$ is assumed to obey the swing equation~(see~\cite{machowski2008power})
\begin{equation}
\label{eq:swing}
m_i\ddot{\theta}_i + d_i(\dot{\theta}_i - \omega^{\mathrm{ref}}) = -\sum_{j=1}^n b_{ij} \sin(\theta_i - \theta_j)+P_i +u_i,~ 
\end{equation}
where $\theta_i$ is the voltage phase angle and $\dot{\theta}_i - \omega^\text{ref}=: \omega_i$ is the frequency deviation at node~$i$, in which $\omega^\text{ref}$ is the nominal frequency (typically 50 Hz or 60 Hz). 
The constants~$m_i$ and $d_i$ are, respectively, the inertia and damping (droop) coefficients, and $P_i$ is the net electric power injected or drawn at node $i$. We denote by $u_i$ the input from the secondary controller, to be introduced shortly. To simplify notation, we have omitted the time-dependence of the states throughout, e.g., $\theta_i(t)$ is denoted $\theta_i$. 

This paper is concerned with a small-signal analysis of the system~\eqref{eq:swing}, and we therefore linearize the system  around the equilibrium where $\theta_i = \theta_j$. Then, by making use of the shifted frequency $\omega_i$ and by defining the droop gain $k_i: = 1/d_i$ and the constant $\tau_i = m_i/d_i$, we obtain the linearized swing dynamics as
\begin{subequations}
\label{eq:swing2}
\begin{align}
\dot{\theta}_i &= \omega_i \\[-3\jot]
\tau_i\dot{\omega}_i &=-\omega_i - k_i \sum_{j=1}^n b_{ij}(\theta_i - \theta_j) + P_i + u_i.
\end{align}
\end{subequations}
The net power injection $P_i$ will, in line with~\cite{Tegling2014, Tegling2016ACC} be modeled as a white stochastic disturbance input that is uncorrelated across the nodes\footnote{This assumption implies that $P$ is a white second-order process with $\mathbb{E}\{P_i(t')\tr{P}_i(t)\} = \delta(t-t')I_i$, where $\delta(t)$ is the Dirac delta function, and without loss of generality we have assumed unit intensity. } (see also Section~\ref{sec:noise}). Thus, $P_i$ captures random fluctuations in generation and load. We remark, however, that the performance analysis performed here is relevant also under other input scenarios, see~\cite{Tegling2014}.

%

\subsection{Secondary frequency control}
A secondary controller input is applied in order to eliminate any stationary frequency errors that arise in standard droop control. This is achieved through integral action. In this paper, we consider two integral controllers that were also analyzed in~\cite{Tegling2016ACC,Andreasson2017CDC}:
\subsubsection{DAPI control}
With distributed averaging PI (DAPI) control, an individual integral state $\Omega_i$ is kept at each node, that can be thought of as tracking the local phase angle deviation. To avoid destabilizing individual drifts in these states, a distributed averaging over them takes place over an additional \emph{communication layer} that is introduced on top of the physical power network layer. This communication layer is modeled by the symmetric, weighted, connected graph $\mathcal{G}^C=(\mathcal{V},\mathcal{E}^C)$, with edge weights given by the interaction strengths $c_{ij} > 0$ for $(i,j) \in \mathcal{E}^C$. 
An illustration is given in Fig. \ref{fig:communicationlayer}. The controller becomes
\begin{subequations}
	\label{eq:DAPI-control-law}
\begin{align}
	u_i^\mathrm{DAPI} &= \Omega_i \\[-2\jot]
	q_i\dot{\Omega}_i &= -\hat{\omega}_i -\sum_{j=1}^n c_{ij} (\Omega_i-\Omega_j),
\end{align}
\end{subequations}
where $\hat{\omega}_i$ is the frequency measured by the controller and $q_i>0$ is a controller gain. Note that $c_{ij}=0$ if $(i,j) \not\in \mathcal{E}^C$. 
\begin{figure}
	\centering
	\includegraphics[width=0.4\textwidth, trim=0cm 7.5cm 0cm 0cm,clip]{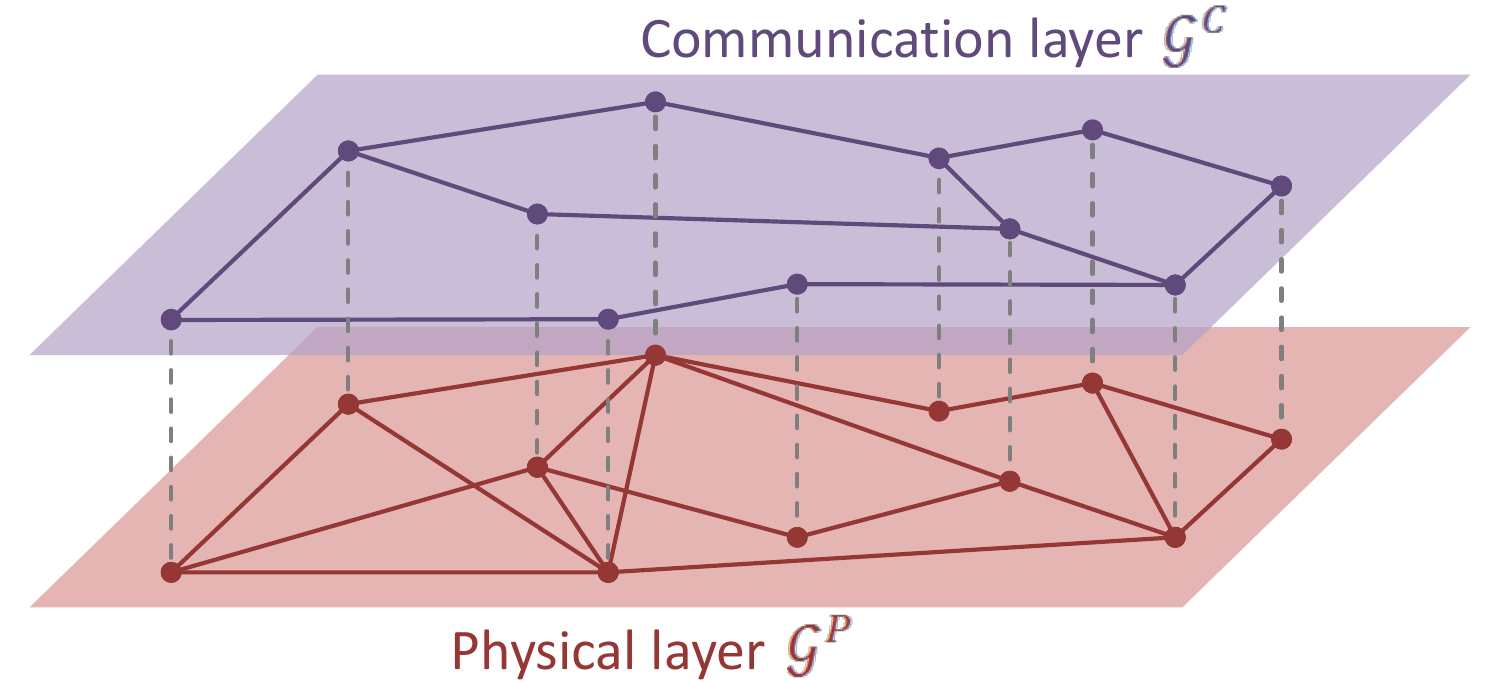}
	\caption{Illustration of the two graph topologies in the DAPI controller. The physical layer $\mathcal{G}^P$ describes the Kron-reduced power network, while the distributed averaging of integral states uses the communication layer $\mathcal{G}^C$.}
	\label{fig:communicationlayer}
\end{figure}

\subsubsection{CAPI control}
In centralized averaging PI (CAPI) control the average phase angle deviation of the entire network is tracked centrally, and subsequently distributed to the individual controllers. The control signal is computed as follows:
\begin{subequations}
	\label{eq:CAPI-control-law}
\begin{align}
	u_i^\mathrm{CAPI}  &= \Omega \\[-2\jot]
	q~\dot{\Omega} &= -\frac{1}{n} \sum_{i=1}^n \hat{\omega}_i
\end{align}
\end{subequations}
where $q>0$ is a constant gain. 

We remark that the DAPI controller corresponds to the CAPI controller in the limit of infinite interaction gains~$c_{ij}$. 

\subsection{Noise model}
\label{sec:noise}
In previous analyses of the performance of secondary frequency control such as \cite{Tegling2016ACC, Andreasson2017CDC, Tegling2017}, it has been assumed that the controller has access to a noiseless measurement of local frequency, so that $\hat{\omega}_i = \omega_i$ in~\eqref{eq:DAPI-control-law} and \eqref{eq:CAPI-control-law}. 
As such, the only disturbance to the system has been~$P_i$, which is due to generation and load fluctuations. In this paper, we will assume imperfect frequency measurements and therefore introduce the additional noise term $\eta_i$, which is modeled as additive white measurement noise on the frequency, so that 
\begin{equation}
\label{eq:noise}
\hat{\omega}_i = \omega_i + \eta_i.
\end{equation}
This noise does, as will see, have a large impact on the relative performance of the secondary controllers.


As already mentioned, we model the power injection fluctuations $P_i$ as uncorrelated white disturbance inputs. We choose to relate its intensity to the white measurement noise~$\eta_i$ through the constant $\varepsilon >0$, so that
\[ \mathbb{E}\{\eta(t') \tr{\eta}_i(t)\} = \varepsilon\mathbb{E}\{P(t')\tr{P}_i(t)\},\] where $P = \tr{\begin{bmatrix}
P_1,\ldots,P_n
\end{bmatrix}}$ and $\eta = \tr{\begin{bmatrix}
\eta_1,\ldots,\eta_n
\end{bmatrix}}$. This allows us to define the process $w \in \mathbb{R}^{2n}$ with $\mathbb{E}\{w(t') \tr{w}(t)\} = \delta(t - t') I$, and construct the input vector 
\begin{align}
\label{eq:input}
	\begin{bmatrix} 
		P \\ \eta 
	\end{bmatrix}  
	= 
	\begin{bmatrix}
		I & 0 \\ 0 & \varepsilon I
	\end{bmatrix} 
	w.
\end{align}


\subsection{Closed-loop system dynamics}
By substituting the secondary controller inputs $u_i^\mathrm{DAPI}$ from~\eqref{eq:DAPI-control-law} and $u_i^\mathrm{CAPI} $~from~\eqref{eq:CAPI-control-law} along with the measurement noise model~\eqref{eq:noise} into the system dynamics~\eqref{eq:swing2}, and by making use of the input vector in~\eqref{eq:input}, we obtain the closed-loop systems on vector form as:
\begin{equation}
\setlength{\arraycolsep}{2pt}
		\begin{bmatrix}
			\dot{\theta} \\ T\dot{\omega} \\ Q\dot{\Omega}
		\end{bmatrix}	
	\!\!	= \!\!
		\begin{bmatrix}
				0 & ~~I & 0 \\ 
				- \!\mat{K}\!\mathcal{L}_{B} & -\!I & I \\ 
				0 & -\!I & -\! \mathcal{L}_{C} \\ 
		\end{bmatrix} \!\!\!
		\begin{bmatrix}
			{\theta} \\ {\omega} \\ {\Omega}
		\end{bmatrix}
		\! + \!
		\begin{bmatrix}
				0 & 0 \\ 
				I & 0 \\
				0 & \!\varepsilon I
		\end{bmatrix} \!\! w~~~~
		\raisetag{0.5\baselineskip}
		\tag{$\mathcal{S}^\DAPI$}
		\label{eq:DAPIclosedloop}
\end{equation}
for the DAPI controlled system, and
\begin{equation}
	\setlength{\arraycolsep}{2pt}
	\begin{bmatrix}
		\dot{\theta} \\ T\dot{\omega} \\ q\dot{\Omega}
	\end{bmatrix}
\!\!	= \!\!
	\begin{bmatrix}
			0 & ~~I & 0 \\ 
			-\!K\!\mathcal{L}_{B} & -\!I & I \\ 
			0 & -\!\frac{1}{n}\mathbf{1}^{\!\!\top} & 0
	\end{bmatrix}\!\!\!
	\begin{bmatrix}
			{\theta} \\ {\omega} \\ {\Omega}
	\end{bmatrix}
		\! + \!
	\begin{bmatrix}
			0 & 0 \\ 
			I & 0 \\
			0 & \varepsilon \mathbf{1}^{\!\!\top} 
	\end{bmatrix}\!\! w~~~~
	\tag{$\mathcal{S}^\CAPI$}
	\label{eq:CAPIclosedloop}
\end{equation}%
for the CAPI controlled system. At this point, we have introduced the state vectors $\theta = \tr{\begin{bmatrix}
\theta_1,\ldots,\theta_n
\end{bmatrix}}$, $\omega = \tr{\begin{bmatrix}
\omega_1,\ldots,\omega_n
\end{bmatrix}}$, and in the DAPI case $\Omega = \tr{\begin{bmatrix}
\Omega_1,\ldots,\Omega_n
\end{bmatrix}}$ (note that $\Omega$ is scalar in the CAPI case), as well as the matrices $K = \mathrm{diag}\{k_i \}$,  $T = \mathrm{diag}\{\tau_i \}$, and $Q = \mathrm{diag}\{q_i \}$. 
We have also made use of the weighted graph Laplacians $\mathcal{L}_B$ and $\mathcal{L}_C$, where $\mathcal{L}_B$ (the susceptance matrix) represents the graph $\mathcal{G}^P$ with weights $b_{ij}$, and $\mathcal{L}_C$ the graph $\mathcal{G}^C$ with weights $c_{ij}$. 

\vspace{0.5mm}
\begin{rem} 
In this model, we have assumed that the frequency $\omega_i$ enters without noise in the system dynamics~\eqref{eq:swing2}. While this is meaningful in a setting with synchronous machines, a power electronic inverter interfacing, e.g., a renewable generation source, may emulate the swing equation using a noisy measurement of the frequency. In this case, the DAPI controlled system would become
\begin{equation}
\setlength{\arraycolsep}{2pt}
\begin{bmatrix}
\dot{\theta} \\ T\dot{\omega} \\ Q\dot{\Omega}
\end{bmatrix}	
\!\!	= \!\!
\begin{bmatrix}
0 & ~~I & 0 \\ 
- \!\mat{K}\!\mathcal{L}_{B} & -\!I & I \\ 
0 & -\!I & -\! \mathcal{L}_{C} \\ 
\end{bmatrix} \!\!\!
\begin{bmatrix}
{\theta} \\ {\omega} \\ {\Omega}
\end{bmatrix}
\! + \!
\begin{bmatrix}
0 & 0 \\ 
I & \varepsilon I \\
0 & \varepsilon I
\end{bmatrix} \!\! w.~~~~
\tag{$\tilde{\mathcal{S}}^\DAPI$}
\label{eq:additional_noise_term}
\end{equation}
This leads to correlations in the noise input between the input channels. 
These correlations do, however, not affect the qualitative behavior of the system (see the appendix for an elaboration). We therefore limit the upcoming analysis to the uncorrelated input modeled in~\eqref{eq:DAPIclosedloop} and~\eqref{eq:CAPIclosedloop}.
%
%
\end{rem}

\subsection{Performance metric}
The performance of the proposed control strategies will be evaluated using the \emph{price of synchrony} performance metric~\cite{Tegling2014}, in order to facilitate a comparison to the analysis in~\cite{Tegling2016ACC}. This metric measures performance in terms of power losses associated with transient power flows. 
For this purpose, we define a performance output $y$ from the systems~\eqref{eq:DAPIclosedloop} and~\eqref{eq:CAPIclosedloop} as
\begin{align}
\label{eq:output}
	y := \mathcal{L}_G^{1/2}\theta
\end{align}
where the conductance matrix $\mathcal{L}_G$ is the weighted graph Laplacian of the graph $\mathcal{G}^P$ with weights $g_{ij}\ge 0$, and where $\mathcal{L}_G^{1/2}$ is the unique positive semidefinite square root of $\mathcal{L}_G$.

Using this output definition, the squared Euclidean norm of the performance output $\tr{y}y = \tr{\theta} \mathcal{L}_G \theta = \sum_{(i,j) \in \mathcal{E}^P} g_{ij}(\theta_i - \theta_j)^2 $ constitutes the approximate instantaneous power loss of the entire system. Based on that, we obtain a performance measure by applying the $\mathcal{H}_2$ norm; namely, for an input-output stable system $\mathcal{S}$ subject to a white stochastic input $w$ it holds that
\begin{align}
\label{eq:h2norn}
	||\mathcal{S}||_2^2 = \lim\limits_{t\rightarrow \infty} \mathbb{E}\{ \tr{y}(t)y(t) \}.
\end{align}
Hence, the squared $\mathcal{H}_2$ norm can be interpreted as the expected power loss in the presence of persistent disturbances. 

\section{The Impact of Measurement Noise on Performance}
\label{sec:results}
In this section, we compute the performance metric~\eqref{eq:h2norn} for the closed-loop system with the DAPI controller and the CAPI controller, respectively. 
The expression for the squared \hn norm can be split into two parts; one associated with the power injection noise $P$, which was already analyzed in~\cite{Tegling2016ACC}, and one associated with the measurement noise $\eta$. It is the impact of the latter part that is the focus of this paper. 

We begin by stating a number of assumptions under which the closed-form expressions for the \hn norms are derived:
\vspace{0.5mm}
\begin{assumption}[Identical line properties]
All lines have uniform and constant conductance-to-susceptance ratios $\alpha:=\frac{g_{ij}}{b_{ij}}$ such that $\mathcal{L}_G = \alpha \mathcal{L}_B$.
\end{assumption}
\begin{assumption}[Uniform control parameters]
	The generator units and their control parameters are uniform, so that $T=\tau I$, $Q=qI$ and $K=kI$.
\end{assumption}
\begin{assumption}[Communication layer topology] 
	The topology of the communication layer is identical to the physical layer so that $\mathcal{G}^P = \mathcal{G}^C$. The interaction strengths~$c_{ij}$ are characterized through the constant scalar~$\gamma>0$ so that $c_{ij} = \gamma b_{ij}$ for all $(i,j) \in \mathcal{E}^P = \mathcal{E}^C$, yielding $\mathcal{L}_C=\gamma\mathcal{L}_B$.
\end{assumption}
\vspace{1mm}
While these assumptions are made for tractability purposes, they can be physically motivated, see~\cite{Tegling2016ACC}. 

The performance in terms of expected power losses for the DAPI and CAPI controlled systems can now be stated as follows.

\begin{theorem}[Performance under measurement noise]
	\label{prop:mainres}
The \hn norm for the DAPI-controlled system~\eqref{eq:DAPIclosedloop} with the output~\eqref{eq:output} is given by 
\begin{equation}
\label{eq:dapinorm}
||\mathcal{S}^{\DAPI}||_2^2 = ||\mathcal{S}^{\DAPI}_P||_2^2 + ||\mathcal{S}^{\DAPI}_\eta||_2^2, 
\end{equation}
\vspace{-0.2mm}
where 
\vspace{-0.2mm}
\begin{equation}  
\label{eq:dapiPnorm}
||\mathcal{S}^{\DAPI}_P||_2^2 = \frac{\alpha}{2k}
\sum_{i=2}^{n}
\frac{1}{1+ \varphi(\lambda_i,\gamma,k,q,\tau)^{-1}}
\end{equation}
are the expected power losses associated with the power injection noise $P$ and
\begin{align}
	||\mathcal{S}^{\DAPI}_\eta||_2^2 = \varepsilon^2	\frac{\alpha}{2k}\sum_{i=2}^{n}
	\frac{1}{ \gamma\lambda_i}\cdot\frac{1}{1+
	\varphi(\lambda_i,\gamma,k,q,\tau)}
		\label{eq:H^eta}
\end{align}
are the expected losses associated with the frequency measurement noise $\eta$. The function $\varphi$ is defined as 
\begin{equation}
\varphi(\lambda_i,\gamma,k,q,\tau):= \frac{kq^2 \lambda_i + q \gamma \lambda_i + \tau (\gamma \lambda_i)^2}{q + \tau\gamma\lambda_i },
\end{equation}
where $\lambda_i$ with $ 0 = \lambda_1<\lambda_2 \le \ldots \le \lambda_n$ are the eigenvalues of $\mathcal{L}_B$.
The \hn norm for the CAPI-controlled system~\eqref{eq:CAPIclosedloop} is given by
\begin{equation*}
||\mathcal{S}^{\CAPI}||_2^2 = ||\mathcal{S}^{\CAPI}_P||_2^2 = \frac{\alpha}{2k}\left(n-1\right)
, \end{equation*}
and is independent of any measurement noise $\eta$. 
\end{theorem}
\begin{proof}
The proof is omitted since it is analogous to the proof of~\cite[Theorem 3.2]{Tegling2016ACC}.	The fact that the \hn norm can be partitioned into one part associated with the input $P$ and one part associated with the input $\eta$ is due to these inputs being uncorrelated. The total \hn norm is therefore the sum of the contributions of each of these inputs. 
	\end{proof}
\vspace{0.5mm}
\begin{rem}
	The expression~\eqref{eq:H^eta} is positive and represents the additional power losses arising due to the frequency measurement noise~$\eta$ in the DAPI controlled system. In the CAPI case, $\eta$ does not give rise to any additional power losses. This can be explained by the fact that the CAPI controller affects all generators equally. Any error caused by the measurement noise may affect the synchronous frequency, but will not induce additional power flows between generators.
\end{rem}
\vspace{0.5mm}

\noindent By noting that the function $\varphi_i := \varphi(\lambda_i,\gamma,k,q,\tau)$ is positive, it is readily verified that in the absence of frequency measurement noise~$\eta$, the losses associated with DAPI control are strictly smaller than those associated with CAPI control: \[||\mathcal{S}^{\DAPI}_P||_2^2 <||\mathcal{S}^{\CAPI}_P||_2^2.\] It can also be shown by analyzing $\varphi_i$ that this loss reduction a) benefits from weak distributed averaging, i.e., a small value for $\gamma$ 
and b) is greater for sparsely interconnected network topologies~\cite{Tegling2016ACC}.
We will now discuss how the additional term~$||\mathcal{S}^{\DAPI}_\eta||_2^2$ affects these conclusions.

\subsection{The role of distributed averaging}
The parameter $\gamma$ characterizes the interaction strengths $c_{ij}$ in the distributed averaging filter of the DAPI controller. It is a tunable parameter that determines how closely the integral states~$\Omega_i$ for $i \in \mathcal{V}$ are kept together. A large $\gamma$ increases the information flow through the network and speeds up the distributed averaging of the integral states, meaning that these states follow each other more closely.  In the limit where $\gamma \rightarrow \infty$, the integral states are identical and we have retrieved the CAPI algorithm. The following corollary to Proposition~\ref{prop:mainres} is easy to show:
\begin{corollary} 
\[ \lim_{\gamma \rightarrow \infty} ||S^\mathrm{DAPI}||_2^2 = ||S^\mathrm{CAPI}||_2^2.\]
It also holds that $\lim_{\gamma \rightarrow \infty} ||S^\mathrm{DAPI}_\eta||_2^2 = 0$, i.e., the losses associated with measurement noise tend to zero as $\gamma \rightarrow \infty$.
\end{corollary}
\vspace{0.4mm}

Small values on $\gamma$ have the opposite effect. They allow for local frequency deviations to be handled more by the local controller, as information about local phase angle deviations will propagate slowly from node to node. In the absence of measurement noise~$\eta$, a relatively small value for $\gamma$ or, under certain conditions even $\gamma = 0$, has also been shown to optimize DAPI performance, as seen from Lemma~\ref{lem:gammastar}.

\begin{lemma}
	\label{lem:gammastar}
The value $\gamma^\star$ that minimizes $||\mathcal{S}^\mathrm{DAPI}_P||_2^2$ lies in the interval 
$ 0 \le \gamma^\star \le  \max_i \frac{q\sqrt{\tau k \lambda_i } -q}{\tau \lambda_i} .$
If the droop gain is such that $\lambda_i k \tau \le 1$ for all $i = 1,\ldots,n$, it holds $\gamma^\star =0$.
\end{lemma}
\begin{proof}
Follows~\cite[Theorem 3]{Andreasson2017CDC}.
\end{proof}
\vspace{0.4mm}

However, when measurement noise $\eta$ is added to the model, it is immediately obvious that $\gamma = 0$ is never an optimal, or even feasible choice (as predicted also by the stability analyses in~\cite{Andreasson2014ACC, SimpsonPorco2013}). 
The following corollary to Proposition~\ref{prop:mainres} follows directly from~\eqref{eq:H^eta}.
\begin{corollary} 
\label{cor:S-DAPI-gamma-0}	It holds that \[ \lim_{\gamma \rightarrow 0} ||S^\mathrm{DAPI}_\eta||_2^2 = \infty\]%
and therefore $\lim_{\gamma \rightarrow 0} ||S^\mathrm{DAPI}||_2^2 = \infty$.
\end{corollary}
\vspace{1mm}
It turns out that an analytic expression for the optimal choice of $\gamma$ in the presence of measurement noise is difficult to obtain even for special graph structures, and a numerical evaluation for each case is necessary. We can, however, give the following proposition.
\vspace{0.3mm}
\begin{proposition}
\label{prop:gammas}
	Let $\gamma^{\star,\eta}$ be optimal with respect to the power losses~\eqref{eq:dapinorm} in the presence of measurement noise $\eta$, and let $\gamma^\star$ be optimal in its absence (i.e. let $\gamma^\star$ minimize $||S^\mathrm{DAPI}_P||_2^2$). Then,
	\vspace{-0.2mm} \[
	\gamma^{\star,\eta} > \gamma^\star.\] 
\end{proposition}
	\vspace{0.3mm}
\begin{proof}
The case of $\gamma^\star=0$ is trivial due to \ref{cor:S-DAPI-gamma-0}. Assume now $\gamma^\star>0$. Define $\bar{\chi}_i(\gamma)=\frac{1}{1+\varphi_i(\gamma)}$ and $\underline{\chi}_i(\gamma)=\frac{1}{1+\inv{\varphi_i}(\gamma)}$, and note that $\underline{\chi}_i+\bar{\chi}_i=1$. Therefore $\bar{\chi}_i$ (that occurs in the expression for $||\mathcal{S}^\mathrm{DAPI}_\eta ||_2^2$) is increasing whenever $\underline{\chi}_i$ (that occurs in $||\mathcal{S}^\mathrm{DAPI}_P ||_2^2$) is decreasing and vice versa. Moreover, since $\gamma^\star$ is a unique minimizer for $||\mathcal{S}^\mathrm{DAPI}_P||_2^2$ \cite{Tegling2016ACC}, $\sum_{i=2}^n\bar{\chi}_i$ in~$||\mathcal{S}^\mathrm{DAPI}_\eta ||_2^2$ will have a unique maximum at $\gamma^\star$. We also know that $0<\sum_{i=2}^n\bar{\chi}_i(\gamma)<\sum_{i=2}^n\bar{\chi}_i(\gamma^\star)$ for any $\gamma$ in the open interval $(0,\gamma^\star)$ and $0=\lim_{\gamma \rightarrow \infty}\sum_{i=2}^n\bar{\chi}_i(\gamma^\star)<\sum_{i=2}^n\bar{\chi}_i(\gamma)\leq\sum_{i=2}^n\bar{\chi}_i(\gamma^\star)$. That means that the set of all possible values that $\sum_{i=2}^n\bar{\chi}_i(\gamma)$ can assume in the interval $(0,\gamma^\star)$ are also contained in the set of values that it can assume in $(\gamma^\star,\infty)$. The same argument can be applied to $\sum_{i=2}^n\underline{\chi}_i(\gamma)$. This means that for any $\gamma^-$ in the open interval $(0,\gamma^\star)$, there exists a $\gamma^+>\gamma^\star$ such that $\sum_{i=2}^n\bar{\chi}_i(\gamma^-)=\sum_{i=2}^n\bar{\chi}_i(\gamma^+)$ while also $\sum_{i=2}^n\underline{\chi}_i(\gamma^-)=\sum_{i=2}^n\underline{\chi}_i(\gamma^+)$, since $\underline{\chi}_i+\bar{\chi}_i=1$.

From $\frac{\varepsilon^2}{\gamma^-\lambda_i}>\frac{\varepsilon^2}{\gamma^+\lambda_i}$ follows that $\sum_{i=2}^n\frac{\varepsilon^2}{\gamma^-\lambda_i}\bar{\chi}_i(\gamma^-)>\sum_{i=2}^n\frac{\varepsilon^2}{\gamma^+\lambda_i}\bar{\chi}_i(\gamma^+)$ while $\sum_{i=2}^n\underline{\chi}_i(\gamma^-)=\sum_{i=2}^n\underline{\chi}_i(\gamma^+)$ still holds. This means that
$||\mathcal{S}^\mathrm{DAPI}||_2^2\big|_{\gamma = \gamma^-}>||\mathcal{S}^\mathrm{DAPI}||_2^2\big|_{\gamma = \gamma^+}$. Therefore, if we assume $\gamma^-$ locally minimizes $||\mathcal{S}^\mathrm{DAPI}||_2^2$ in the interval $(0,\gamma^\star)$, then $\gamma^+$ would give an even smaller value for $||\mathcal{S}^\mathrm{DAPI}||_2^2$. We conclude that $\gamma^{\star,\eta}\geq\gamma^\star$.

It remains to show that $\gamma^\star$ is not optimal in the presence of $\eta$. Since $\frac{d}{d\gamma}||\mathcal{S}^\mathrm{DAPI}||_2^2\big|_{\gamma = \gamma^\star} = -\frac{a\varepsilon^2}{2k(\gamma^\star)^2}\sum_{i=2}^n \frac{1}{\lambda_i}\frac{1}{1+\varphi_i(\gamma^\star)}<0$, it follows that $\gamma^\star$ is not optimal. Hence, $\gamma^{\star,\eta}>\gamma^\star$.
\end{proof}
\vspace{1mm}
In Fig. \ref{fig:optimal_gamma} we give a numerical example that illustrates the results of this section. 
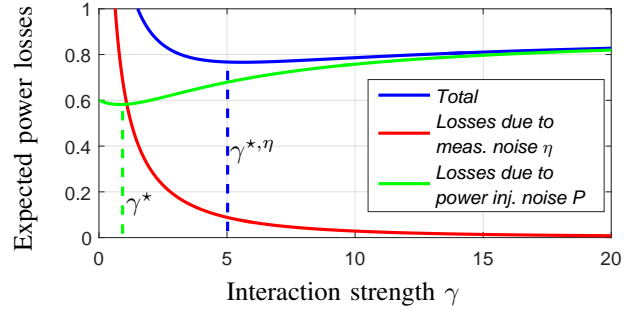
\begin{figure}
	\centering
	\begin{tikzpicture}
		\node[inner sep=0pt] (img) at (0,0)
		{\includegraphics[width=0.5\textwidth,trim=3cm 11.5cm 3cm 11.5cm,clip]{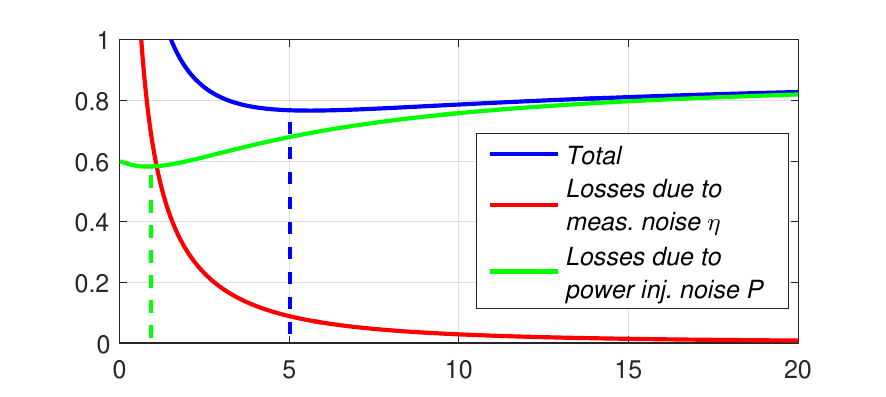}};
		\draw (0, -2.2) node {Interaction strength $\gamma$};
		\node[xshift=-4.5cm,rotate=90,anchor=north]{Expected power losses};
		\draw (-2.7cm, -1cm) node {$\gamma^\star$};
		\draw (-1.2cm, -0.3cm) node {$\gamma^{\star,\eta}$};
	\end{tikzpicture}
	\caption{Power losses' dependence on interaction strength~$\gamma$ in DAPI. 
	Here, we have modeled a complete graph with $n=10$ nodes and set $k=5,~q=\tau=0.8,~\alpha=\varepsilon=1$ and $b_{ij}=0.05$ for all $(i,j)\in \mathcal{E}^P$. One can see that the optimal~$\gamma$ is shifted towards a larger value in presence of noise $\eta$.}
	\label{fig:optimal_gamma}
\end{figure}
Overall, our results imply that the distributed averaging of integral states that takes place in the DAPI controller is increasingly important in the presence of measurement noise. Its relative importance, of course, depends on the noise intensity~$\varepsilon$. In Section~\ref{sec:existence} we will also discuss how this is impacted by the size of the network.

\subsection{The role of network density}
\emph{Network density }captures the notion of the connectivity of a network, and can formally be defined as the proportion of a graph's actual edges to its potential edges, i.e. $|\mathcal{E}^P|/|\mathcal{V}\times\mathcal{V}|$~\cite{Diestelbook}. If this number is small, we call the graph \textit{sparse} and if it is big, we call it dense, or well-interconnected.  

While this property does not effect the expected power losses under CAPI control, it plays an important role for the DAPI control law. In \cite{Tegling2016ACC} it was shown that the expected power losses associated with load disturbances~$P$ increase with the network density. This implies that, in the absence of~$\eta$, the largest relative performance improvement of DAPI control over CAPI control is achieved for sparse network topologies. This is no longer the case in the presence of measurement noise $\eta$, as the losses due to such noise are highest for sparse topologies and decrease with increased network density. 
 Consider the following proposition.
 \vspace{0.3mm}
\begin{proposition}
\label{prop:density}
	Adding an edge to the network $\mathcal{G}^P$, or increasing its susceptance (decreasing the reactance) can only decrease the expected power losses associated with the measurement noise~$||\mathcal{S}^\mathrm{DAPI}_\eta||_2^2$ in~\eqref{eq:H^eta}, and vice versa. 
\end{proposition}
	\vspace{0.3mm}
\begin{proof}
	Let $\mathcal{G}^P=(\mathcal{V},\mathcal{E}^P)$ be the original graph with weights $b_{ij}$ for $(i,j)\in \mathcal{E}^P$. Since $\frac{d\varphi_i}{d\lambda_i}= \frac{2\lambda_i \gamma^2 \tau q+\gamma q^2 + k q^3}{(q+\tau q \lambda_i)^2}>0$, we derive that $\frac{d}{d\lambda_i} ||\mathcal{S}^\DAPI_\eta||_2^2<0$ for all $i = 2,\ldots,n$. Thus, it is sufficient to show that at least one eigenvalue $\lambda_i$ increases while the others do not decrease.
	
	(Part 1: Additional edge) Let $e$ be an additional edge that is added to $\mathcal{G}^P=(\mathcal{V},\mathcal{E}^P)$, constituting the extended graph $\tilde{\mathcal{G}}=(\mathcal{V},\mathcal{E}^P\cup e)$. The eigenvalues satisfy $0=\lambda_1(\mathcal{G}^P)=\lambda_1(\tilde{\mathcal{G}})\leq\dots\leq\lambda_i(\mathcal{G}^P)\leq\lambda_i(\tilde{\mathcal{G}})\leq\dots\leq\lambda_n(\mathcal{G}^P)\leq\lambda_n(\tilde{\mathcal{G}})$ where at least one inequality is strict~\cite[Theorem 3.2]{Mohar91thelaplacian}. The same argument can be used in case an edge is removed.
	
	(Part 2: Increased susceptance) Assume a positive change $\Delta b>0$ for an arbitrary edge $e' =(i,j) \in \mathcal{E}$. Define the graph $\mathcal{G}'=(\mathcal{V},\{e'\})$ where $e'$ has the weight $\Delta b$. The Laplacian of $\mathcal{G}'$ is $\mathcal{L}'$ and that of the original graph $\mathcal{G}^P$ is~$\mathcal{L}_B$. According to the Courant-Weyl inequalities~\cite[Section 2.8]{brouwer12}, it holds $\lambda_i(\mathcal{L}_B+\mathcal{L}')\geq\lambda_i(\mathcal{L}_B)$. Again, at least one inequality must be strict, following
	~\cite[Theorem 2.6.1]{brouwer12}. Reducing a susceptance follows the same argument. 
\end{proof}
\vspace{0.3mm}
\begin{rem}
Proposition~\ref{prop:density} does not require conductances~$g_{ij}$ to be constant, but also holds if the ratio $\alpha=\frac{g_{ij}}{b_{ij}}=\textrm{const.}$ or increasing. 
\end{rem}
\vspace{0.8mm}

This result implies that the \emph{total} losses' dependence	 on the network density is not straightforward in the presence of measurement noise, since the two loss components have opposite dependencies. The best network topology for loss reduction thus depends on remaining system parameters and, in particular, the relative noise intensity characterized by~$\varepsilon$.




\section{Limitations to the Scalability of DAPI~Control}
\label{sec:scalability}
In the previous section, we discussed how the performance in terms of resistive power losses depends on the system parameters for a given network. Aside from the importance of designing the controller so that the losses for a given network are acceptable, an important issue is the \emph{scalability} of a given control design. That is, the ability to extend the network without degrading performance and without needing to alter the network or control design. As we shall see, measurement noise affects the scalability of DAPI control.


\subsection{Scaling of power losses with network size }
To discuss the scalability of the DAPI controller, we need to assume that all parameters and gains of the controller~ \eqref{eq:DAPI-control-law} are fixed. That is, they do not change as the network size grows. Further, we will consider the transient power losses from~\eqref{eq:h2norn} \emph{normalized} by the total number of nodes~$n$. These \textit{per-node} losses should remain bounded in order for the controller to be regarded as scalable. 

The expected power losses associated with measurement noise are given in \eqref{eq:H^eta}. Notice the factor $\frac{1}{\lambda_i}$, which will tend to infinity for small eigenvalues. This causes an unfavorable scaling of the losses in sparse networks. In regular lattice networks, it is possible to derive exact expressions for the asymptotic (in network size) scalings of the losses. More precisely, we will consider 1- and 2-dimensional lattice networks and their $q$-fuzzes, where a 1D lattice corresponds to a path graph. The $q$-fuzz of a lattice graph is created by adding edges between all nodes that are at a graph distance of $q$ or less from each other. We note that scalings for lattices also apply to any graph that can be embedded in such lattices, but defer a detailed discussion to an extended study. 

Consider the following proposition: 
\vspace{0.3mm}
\begin{proposition}[Performance scaling in lattices] 
\label{prop:scaling}
Let the graph~$\mathcal{G}^P$ be a lattice or its $q$-fuzz in one or two dimensions ($d = 1$ or $d = 2$), and let the line susceptances~$b_{ij}$ be bounded. Then, the \emph{per-node} losses associated with measurement noise scale according to:
\[ \frac{1}{n}|| \mathcal{S}^{\mathrm{DAPI}}_\eta||_2^2 \sim  \begin{cases} \varepsilon^2 n  & \text{if} ~~d = 1 \\[-1\jot] \varepsilon^2 \log n & \text{if} ~~ d = 2, \end{cases} \]
where the notation $u(n) \sim v(n)$ implies that $\underline{c}v(n) \le u(n) \le \bar{c}v(n)$ for sufficiently large $n$, where $\underline{c},~\bar{c}$ are positive constants. The per-node losses associated with power injection noise are, on the other hand, upper bounded as: 
\vspace{-1.3mm}
\[ \frac{1}{n}|| \mathcal{S}^{\mathrm{DAPI}}_P||_2^2 \le \frac{\alpha}{2k}.\]
\end{proposition} 
\vspace{0.6mm}
\begin{proof}
First we need to show that the factors $\frac{1}{1+\varphi_i}$ from~\eqref{eq:H^eta} are uniformly bounded. Recall that $\frac{d\varphi_i}{d\lambda_i} > 0$ for $\lambda_i\geq0$. Thus, $\lambda_\textrm{min} = 0$ and $\lambda_\textrm{max} = \max_i \lambda_i$ give the lower and upper bounds for $\varphi_i$, respectively. In a 1D lattice the largest possible eigenvalue is $4b_{\max}$ and in a 2D lattice it is $16b_{\max}^2$, where $b_{\max} = \max_{(i,j) \in \mathcal{E}^P}b_{ij}$, since the eigenvalues of an unweighted path graph are $\lambda_i = 2(1-\cos\frac{\pi i }{n})$ \cite{brouwer12}, so $0 \le \varphi_i \le \varphi(\lambda_{\max}, \gamma, k,q,\tau)$ and 
 $\frac{1}{1+\varphi_i}$ is therefore uniformly bounded with respect to~$n$. 
  Now, we can bound $\frac{1}{n}|| \mathcal{S}^\DAPI_\eta||_2^2$ as: $\underline{c}\frac{\varepsilon^2}{n}\sum_{i=2}^n \frac{1}{\lambda_i^1}:=\frac{\alpha\varepsilon^2}{2kn\gamma b_{\max} (1+\varphi(\lambda_\textrm{max}))}\sum_{i=2}^n \frac{1}{\lambda_i^1} \leq \frac{1}{n}|| \mathcal{S}^\DAPI_\eta||_2^2 \leq \frac{\alpha\varepsilon^2}{2kn\gamma b_{\min} (1+\varphi(\lambda_\textrm{min}))}\sum_{i=2}^n \frac{1}{\lambda_i^1} =: \bar{c}\frac{\varepsilon^2}{n}\sum_{i=2}^n \frac{1}{\lambda_i^1}$, where $\lambda_i^1$ are the Laplacian eigenvalues of the graph where all weights are equal to 1.
  According to \cite{Klein1993}, the relationship  $\sum_{i=2}^n \frac{1}{\lambda_i^1}=\frac{1}{n}K_f$ holds, where $K_f$ is the Kirchoff index of a graph. We can then write $\frac{\underline{c}\varepsilon^2}{n^2}K_f \leq \frac{1}{n}|| \mathcal{S}^\DAPI_\eta||_2^2 \leq \frac{\bar{c}\varepsilon^2}{n^2}K_f$. It is proven in \cite{Barooah} that the Kirchhoff index for infinite lattices, including $q$-fuzzes, scales like $K_f \sim n^3$ if $d = 1$ and $K_f \sim n^2 \log n$ if $d = 2$. Thus, we get $\frac{1}{n}|| \mathcal{S}^{\mathrm{DAPI}}_\eta||_2^2 \sim \varepsilon^2 n$ if $d=1$ and $\frac{1}{n}||\mathcal{S}^{\mathrm{DAPI}}_\eta||_2^2 \sim \varepsilon^2 \log n$ if $d=2$.
%
\end{proof}
\vspace{1mm}

This result means that while the losses, when evaluated per node, were upper bounded for DAPI in the absence of noise, they may grow unboundedly in large 1- or 2-dimensional lattice networks in the presence of noise. Fig.~\ref{fig:DAPI-scaling} provides a numerical example.
This unbounded growth of the power losses per generator can be understood as caused by the secondary control input becoming increasingly distorted as the network grows. In practice, however, there is clearly a limit on how large the transient losses can become, which depends on the generators' power ratings. The scaling result in Proposition~\ref{prop:scaling} should therefore be interpreted as setting a limit on the feasible network size for each controller tuning. 

\begin{figure}
	\centering
	\begin{tikzpicture}
	\node[inner sep=0pt] (img) at (0,0)
	{\includegraphics[width=0.5\textwidth,trim=3.5cm 12cm 3cm 12cm,clip]{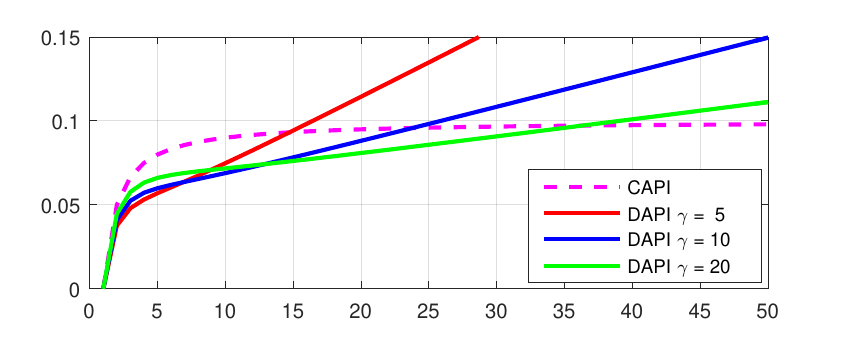}};
	\draw (0, -2) node {Number of nodes $n$};
	\node[xshift=-4.5cm, yshift=0cm ,rotate=90,anchor=center]{Per-node power losses};
	\end{tikzpicture}
	\caption{Per node expected power losses~$\frac{1}{n}||\mathcal{S}||_2^2$ as a function of the network size in a ring graph with CAPI and DAPI for different interaction strengths~$\gamma$. 
	Here, we have set $k=5,~q=\tau=0.8,~\alpha=1,~\varepsilon=0.5$ and $b_{ij}=0.1$, $(i,j)\in \mathcal{E}^P$.}
	\label{fig:DAPI-scaling}
\end{figure}

\subsection{Existence of optimal control design}
\label{sec:existence}
In contrast to previous work, where it was shown that the DAPI strategy always leads to superior performance in comparison to the CAPI approach, this paper's results show that this no longer necessarily applies in the presence of measurement noise. Due to the unfavorable scaling behavior in sparse networks, the expected losses under DAPI control can vastly exceed those under CAPI control. We now show that this, however, can be counteracted by a proper adjustment of the distributed averaging gain parameter $\gamma$.
\vspace{0.3mm}

\begin{proposition}
For any network $\mathcal{G}^P$ it holds that 
\begin{align}
\label{eq:gammaex}
||\mathcal{S}^\DAPI||_2^2 < ||\mathcal{S}^\CAPI||_2^2 ~,~~ \mathrm{if}~~ \hat{\gamma} < \gamma < \infty,
\end{align}
where $\hat{\gamma} = \frac{\varepsilon^2}{\lambda_2}$, and $\lambda_2$ is the smallest non-zero eigenvalue of $\mathcal{L}_B$.  
\end{proposition}

\begin{proof}
The DAPI \hn norm~\eqref{eq:dapinorm} can be written $||\mathcal{S}^{\DAPI}||_2^2$
$=\frac{\alpha}{2k} \sum_{i=2}^n \left( \frac{1}{1+\varphi_i^{-1}}+\frac{\varepsilon^2}{\gamma\lambda_i} \frac{1}{1+\varphi_i} \right)$. This expression can be upper bounded by $(n-1)$ times the largest summand: $||\mathcal{S}^{\DAPI}||_2^2
\leq \frac{\alpha}{2k} (n-1) \cdot \max_{\lambda_i>0}\left\{ \frac{1}{1+\varphi_i^{-1}}+\frac{\varepsilon^2}{\gamma\lambda_i} \frac{1}{1+\varphi_i} \right\}$. Let $\lambda^\star$ be the maximizing eigenvalue.
We are looking to choose $\gamma$ so that $||\mathcal{S}^\DAPI||_2^2 < ||\mathcal{S}^\CAPI||_2^2 = \frac{\alpha}{2k} (n-1)$ holds. Some simplifications reveal that this holds if $\gamma > \frac{\varepsilon^2}{\lambda^\star}  \geq \frac{\varepsilon^2}{\lambda_2} =: \hat{\gamma}$, and \eqref{eq:gammaex} follows.
\end{proof}
\vspace{1 mm}
Moreover, we already established that there exists a value~$\gamma^\star$ that minimizes the losses for the DAPI controller.  Proposition~\ref{prop:scaling} implies that at this minimum (and for larger $\gamma$), DAPI will perform better than CAPI. Thus, the CAPI performance can still be seen as an upper bound for DAPI. Nonetheless, one should be aware that selecting a too small $\gamma$ in DAPI can lead to much worse performance. 

This result shows that despite the unfavorable scaling of losses due to noise in certain topologies, a distributed strategy is not necessarily worse than the centralized CAPI strategy (it is actually better, for some optimal configuration). However, since $\hat{\gamma}$ depends on $\lambda_2$, it unfortunately means that the choice of $\gamma$ cannot be made independently of the network size for all topologies, and therefore, the algorithm is less scalable. For a very large and sparse network, one would have to revert to a CAPI algorithm as $\hat{\gamma} \rightarrow \infty$.

\section{Conclusions}
\label{sec:conclusions}

We have investigated the performance limitations that arise due to noisy frequency measurements in a distributed secondary frequency control law, namely DAPI. Our main conclusion is that noise may have a large impact on performance, creating a need for very careful controller tuning, whereas the performance of the corresponding centralized controller, CAPI, remains unaffected by noise. 
In principle, our results state that the distributed averaging filter in DAPI carries an increased importance when measurements are noisy -- larger gains and higher network density will reduce the noise impact. We prove that there is an optimal configuration of this filter that allows DAPI to still perform better than CAPI. Therefore, the main conclusions from~\cite{Tegling2016ACC, Tegling2017, Andreasson2017CDC} still hold. A poor configuration on the other hand, can lead to much worse performance in DAPI than in CAPI, and in very large and sparse networks, the poor scalability of the DAPI controller may make it sensible to revert to CAPI.


A relevant extension to this work is to consider separate network topologies for the physical layer and the communication layer in DAPI.
Preliminary analyses reveal that the unfavorable performance scaling observed in sparse networks has bearing on the communication network layer. Therefore, increasing the density in the communication layer alone can be expected to improve performance, in a way similar to increasing interaction strengths. A detailed analysis is part of ongoing work. 

A promising alternative to counteract the performance limitations in DAPI are phasor measurement units (PMUs). These measure phase directly and thereby eliminate the need for integral states that are distorted by noisy frequency measurements. Since the availability of PMUs, however, is likely to be limited, an important question that is subject to ongoing research is how strategical placement of PMUs in the network may mitigate the impact of measurement errors. 


\appendix

\subsection{Alternative noise model}\label{alternative-noise-model}
If the additional noise term enters the dynamic equation as in \eqref{eq:additional_noise_term}, the \hn norm becomes $||\mathcal{\tilde{S}}^{\DAPI}||_2^2 =\frac{a}{2k} \sum_{i=2}^{n}\left( (1+\varepsilon^2)\frac{1}{1+\varphi_i^{-1}} + \varepsilon^2 \left(2+\frac{1}{\gamma \lambda_i} \right) \frac{1}{1+\varphi_i}\right)$. Clearly, the factor 
$(1+\varepsilon^2)$ does not change the qualitative behavior of the first term. 
For the second term it is the scaling in~$n$ that is of most interest~(see Section~\ref{sec:scalability}). This depends on the factor~$\frac{1}{\gamma\lambda_i}$ and the qualitative result of Proposition~\ref{prop:scaling} is not affected by adding the term $2$ to this factor. 


\bibliographystyle{IEEEtran}
\bibliography{EmmasBib15}
\end{document}